\documentclass[reqno,12pt]{amsart}
\usepackage{a4wide,amsfonts,amsmath,latexsym,amssymb,euscript,graphicx}

\newtheorem{theorem}{Theorem}[section]

\newtheorem{lemma}[theorem]{Lemma}
\newtheorem{corollary}[theorem]{Corollary}
\newtheorem{proposition}[theorem]{Proposition}
\newtheorem{example}[theorem]{Example}
\newtheorem{remark}[theorem]{Remark}
\numberwithin{equation}{section}

\begin{document}

\begin{center}
{\Large{\bf Two properties of vectors of quadratic forms in Gaussian random variables}}
\normalsize
\\~\\ Vladimir I. Bogachev (Moscow State University)\\
 Egor D. Kosov (Moscow State University)\\
 Ivan Nourdin (Universit\'e de Lorraine)\\
 Guillaume Poly (Universit\'e du Luxembourg)\\
\end{center}

\vskip .2in

{\small
\noindent{\bf Abstract}:
We study distributions of random vectors
whose components are second order polynomials in Gaussian random variables.
Assuming that  the law of such a vector is not absolutely continuous with
respect to Lebesgue measure, we derive some interesting consequences. Our second result gives a
characterization of limits in law for sequences of such vectors.\\


}

\section{Introduction and main results}

Due to many applications in probability and statistics,
quadratic forms or more general second order polynomials
in Gaussian random variables are an object of great importance.
The aim of this paper is to present some new results about distributions  of random vectors
whose components are such quadratic forms.

To be more specific, let us
fix an integer number $k\ge  2$ and let us
 introduce an array $g_{i,j}$, $1\le i\le k$, $j=1,2,\ldots$ of $N(0,1)$ random variables. Assume
 that the variables $g_{i,j}$ are jointly Gaussian
and that ${\rm E}[g_{i,j}g_{i,j'}]=0$ whenever $j\neq j'$
(that is, for any fixed $i$, the sequence $g_{i,1},g_{i,2},\ldots$ is composed of independent $N(0,1)$ random variables).
Let us also consider a random vector $F=(F_1,\ldots,F_k)\in\mathbb{R}^k$ whose components are
``quadratic forms'', that is,
for any $i=1,\ldots,k$,
\begin{equation}\label{F}
F_i = \sum_{j=1}^\infty \lambda_{i,j}(g_{i,j}^2-1)\quad\mbox{for some $\lambda_{i,j}\in\mathbb{R}$ with $\sum_{j=1}^\infty \lambda_{i,j}^2<\infty$}.
\end{equation}

In his seminal paper \cite{Kus}, Kusuoka showed that
the law of  $F=(F_1,\ldots,F_k)$ as above is not absolutely continuous with respect to Lebesgue measure
if and only if there exists a nonconstant polynomial
$P$ on $\mathbb{R}^k$ such that $P(F_1,\ldots,F_k)=0$ almost surely
(see also \cite[Theorem 3.1]{NNP}, where it is further shown that the degree of $P$ can always be
chosen less than or equal to $k2^{k-1}$).
But what can be said about the usual {\it linear} dependence of $F_1,\ldots,F_k$?
One of the goals of this paper is to provide a number
of positive and negative results in the  spirit of the following theorem, which
will be proved (actually in a more general framework) in Section \ref{Sec2}.

\begin{theorem}\label{main1}
Let $k\in\mathbb{N}$ be fixed. Let $F=(F_1,\ldots,F_k)$ be a random vector given by {\rm(\ref{F})}
 such that the law of $F$ is not absolutely continuous.
Then, there exist $k-1$ independent $N(0,1)$ random variables $\eta_1,\ldots,\eta_{k-1}$ and
 $2k-1$ real numbers $a_1,\ldots,a_k,b_1,\ldots,b_{k-1}$
 such that $(a_1,\ldots,a_k)\not=(0,\ldots,0)$ and almost surely
$$
a_1 F_1 + \cdots + a_k F_k = b_1 (\eta_1^2-1) + \cdots+ b_{k-1}(\eta_{k-1}^2-1).
$$
\end{theorem}

In order to better understand the significance of this theorem, let us comment on it a little bit.
Let its assumptions hold. In case $k=1$ this is only possible if $F_1=0$.
When $k=2$, it can be shown (see, e.g., \cite[Proposition 3.2]{NNP}) that one can actually choose $b_1=0$,
which means that $F_1$ and $F_2$ are necessarily linearly dependent.
When $k\ge  3$, the situation becomes more difficult.
It is no longer true that the variables $F_i$ are necessarily linearly dependent
when $F$ is not absolutely continuous, as the following simple counterexample
shows. Let $g_1,g_2\sim N(0,1)$ be independent. Set $F_1=g_1^2-1$, $F_2=g_2^2-1$ and
$$
F_3=g_1g_2=\frac{1}{2\sqrt{2}}\biggl[\Bigl(\Bigl(\frac{g_1+g_2}{\sqrt{2}}
\Bigr)^2 - 1\Bigr)
-\biggl(\Bigl(\frac{g_1-g_2}{\sqrt{2}}
\Bigr)^2 - 1\Bigr)\biggr].
$$
The second equality just shows that $F_3$ is indeed of the form (\ref{F}).
It is readily checked that the covariance matrix of $F_1,F_2,F_3$ is not degenerate
(hence $F_1, F_2, F_3$ are  linearly independent), although
the law of $(F_1,F_2,F_3)$ is obviously not absolutely continuous, since
$F_3^2-(F_1+1)(F_2+1)=0$.
Therefore,  the best one can achieve about the linear dependence is precisely
what we state in our theorem.

Let us now discuss the second main result of this paper, which is a description
of the possible limits in law for {\it sequences} of
vectors of quadratic forms.
To be in a position to state a precise result, we first need to introduce some $n$
in the previous framework. So, fix an integer number $k\ge  1$ and
let now $g_{i,j,n}$, $1\le i\le k$, $j,n=1,2,\ldots$, be an array of $N(0,1)$ random variables.
Assume that, for each fixed $n$, the variables $g_{i,j,n}$ are jointly Gaussian
and that ${\rm E}[g_{i,j,n}g_{i,j',n}]=0$ whenever $j\neq j'$.
Let us also consider a sequence of random vectors $F_n=(F_{1,n},\ldots,F_{k,n})\in\mathbb{R}^k$
whose components are again ``quadratic forms'':
for any $i=1,\ldots,k$ and any $n\ge  1$,
$$
F_{i,n} = \sum_{j=1}^\infty \lambda_{i,j,n}(g_{i,j,n}^2-1)
\quad\mbox{for some $\lambda_{i,j,n}\in\mathbb{R}$ with $\sum_{j=1}^\infty \lambda_{i,j,n}^2<\infty$}.
$$
We then have the following theorem, which is our second main result.

\begin{theorem}\label{main2}
Let us assume that $F_n=(F_{1,n},\ldots,F_{k,n})$ converges in law
as $n\to\infty$.
Then the limit law coincides with the distribution of a vector of the form $\eta+F$, where
$F=(F_1,\ldots,F_k)$ is of the form {\rm(\ref{F})} and $\eta=(\eta_1,\ldots,\eta_k)$ is
an independent Gaussian random vector.
\end{theorem}

When $k=1$ (the one-dimensional case), Theorem \ref{main2} was actually shown by Arcones \cite{A99} (see also
Sevastyanov \cite{S61} or Nourdin and Poly \cite{NP}).
At first glance one could be tempted to think that, in order to show Theorem
\ref{main2} in the general case $k\ge  2$, it suffices to apply the Arcones
 theorem along with the Cram\'er--Wold theorem.
But if we try to implement this strategy, we face a crucial issue. Indeed, consider
a random vector $G=(G_1,\ldots,G_k)$ in $\mathbb{R}^k$ and assume that each linear combination of the
variables $G_i$ is a
quadratic form in Gaussian random variables;
how can we deduce from this  that $G$ has the same law as  $F$ given by (\ref{F})?

The paper is organized as follows. Section 2 contains our results related to Theorem \ref{main1}, whereas
the proof of Theorem \ref{main2} is performed in Section 3.

\section{Second order polynomial mappings with linearly dependent Malliavin derivatives}\label{Sec2}

A basic fact of the Malliavin calculus is that finitely many
Sobolev functions $F_1,\ldots,F_n$ on a space with a Gaussian $\mu$
have a joint density of distribution provided that their Malliavin
gradients $D_HF_1,\ldots, D_HF_n$ along the Cameron--Martin space $H$
are linearly independent almost everywhere.
For this reason diverse sufficient conditions for such independence are of interest,
which leads to a natural question about consequences of the alternative situation
where the gradients are linearly dependent on a positive measure set.
One can hardly expect a useful general characterization of this, but
the situation may be more favorable for various special classes of functions,
in particular, for measurable polynomials.
One of the first results in this direction (already mentioned above)  was obtained by Kusuoka~\cite{Kus}.
This result has been recently extended in \cite{NNP}
as follows: the measure induced by $(F_1,\ldots,F_n)$ is not absolutely
continuous precisely when there is a polynomial dependence between
$F_1,\ldots,F_n$, i.e. there is a nonzero polynomial $\psi$ such that
$\psi(F_1,\ldots,F_n)=0$.
But what can be said about usual linear dependence of $F_1,\ldots,F_n$?
Of course, in general there might be no such dependence
even in the finite-dimensional case.
However, there are cases where the linear dependence of derivatives of
mappings on $\mathbb{R}^d$ on a positive
measure set yields the usual linear dependence of the mappings themselves.
This is obviously the case for linear functions and can be also verified
for quadratic forms. The goal of this section is to present a number
of positive and negative results of this sort.
We prove that if $k$ measurable linear mappings $A_1,\ldots,A_k$
on a  space with a Gaussian measure $\mu$ are such that the vectors
$A_1x,\ldots,A_kx$ are linearly dependent for every $x$ in a set of positive
$\mu$-measure, then there is a measurable linear operator $D$ of rank $k-1$ such
that the operators $A_1,\ldots,A_k, D$ are linearly dependent, i.e.
$D=c_1A_1+\cdots+ c_kA_k$ with some numbers $c_1,\ldots,c_k$.
In general, the rank of $D$ cannot be made smaller.
However, if $k=2$ and $A_1$ and $A_2$ are the
second derivatives along $H$ of some second order polynomials, then
the above assertion is true with $D=0$, that is, $A_1$ and $A_2$ are linearly dependent.

Let $\mu$ be a centered Radon Gaussian measure on a locally
convex space $X$ with the topological dual $X^*$, i.e., every functional
$f\in X^*$ is a centered Gaussian random variable on $(X,\mu)$.
Basic concepts and facts related to Gaussian measures can be found in \cite{B1} and \cite{B2}.
We recall some of them.

The Cameron--Martin space of $\mu$ is the
set
$$
H=\{h\in X\colon\, |h|_H<\infty \},
$$
where
$$
|h|_H:=\sup \bigl\{l(h)\colon\ l\in X^*, \, \|l\|_{L^2(\mu)}\le 1\bigr\}.
$$
It is known that the closure of $H$ in $X$ has full measure, so we shall assume throughout
that $H$ is dense. Such a measure $\mu$ is called nondegenerate. An equivalent
condition is that the distribution of every nonzero  $f\in X^*$ has a density.

Let $X_\mu^*$ denote the closure of $X^*$ in $L^2(\mu)$. The elements of $X_\mu^*$ are
called measurable linear functionals on $X$. Such a functional admits a version
that is linear on all of $X$ in the usual algebraic  sense. Conversely,
every $\mu$-measurable function that is algebraically  linear belongs to the space $X_\mu^*$.

It is known that $H$ consists of all vectors $h$ such that $\mu$ is equivalent to its
shift $\mu_h$ defined by $\mu_h(B)=\mu(B-h)$.
It is also known that
every vector $h\in H$ generates a measurable linear functional
$\widehat{h}$ on $X$ such that
$$
(f,\widehat{h})_{L^2(\mu)}=f(h),\quad f\in X^{*}.
$$
Every element in $X_\mu^*$ can be represented in this way,
so that the mapping $h\mapsto \widehat{h}$ is one-to-one.
It is known that
$H$ is a separable Hilbert space with the inner product
$$
(h,k)_H:=(\widehat{h},\widehat{k})_{L^2(\mu)}.
$$

If $\{e_n\}$ is an orthonormal basis in $H$, then
$$
\widehat{h}(x)=\sum_{n=1}^\infty (h,e_n)_H \widehat{e}_n(x),
$$
where the series converges in $L^2(\mu)$ and almost everywhere.
In the case where $\mu$ is the standard Gaussian measure on $\mathbb{R}^\infty$
(the countable power of $\mathbb{R}$ or the space of all real
sequences $x=(x_n)$) and
$\{e_n\}$ is the standard basis in $H=l^2$, we have $\widehat{e}_n(x)=x_n$ and
$$
\widehat{h}(x)=\sum_{n=1}^\infty h_nx_n.
$$
Given a bounded  operator $A\colon\, H\to H$ let $\widehat{A}$ denote the associated
measurable linear operator on~$X$, i.e., a measurable linear mapping from $X$ to $X$
such that $\widehat{h}(\widehat{A}x)=\widehat{A^*h}(x)$ for every $h\in H$.
This operator can be defined by the formula
$$
\widehat{A}x=\sum_{n=1}^\infty \widehat{e}_n(x) Ae_n,
$$
where the series converges in $X$ for almost all $x$ (which is ensured by the Tsirelson theorem).

If $A$ is a Hilbert--Schmidt operator (and only in this case)
the operator $\widehat{A}$ takes values in~$H$. Then
$(\widehat{A}x,h)_H=\widehat{A^{*}h}(x)$ for every $h\in H$
and the above series converges in $H$ for almost all $x$.

The space $L^2(\mu)$ can be decomposed in the orthogonal sum $\bigotimes_{k=0}^\infty \mathcal{X}_k$ of
mutually orthogonal closed subspaces $\mathcal{X}_k$ constructed as follows. Letting
$E_k$ be the closure in $L^2(\mu)$ of polynomials of the form
$f(\xi_1,\ldots,\xi_n)$, where $f$ is a polynomial of order $k$ on $\mathbb{R}^n$ and $\xi_i\in X_\mu^*$,
the space $\mathcal{X}_k$ is the orthogonal complement of $E_{k-1}$ in $E_k$, $E_0=\mathcal{X}_0$ is the
one-dimensional space of constants. For example, $\mathcal{X}_1=X_\mu^*$.
Functions in $E_k$ are called measurable
polynomials of order~$k$.
The elements of $\mathcal{X}_k$ are also
referred to as elements of the homogeneous
Wiener chaos of order~$k$ (although they are not homogeneous polynomials
excepting the case $k=1$).
It has been recently shown in \cite{AY} that the class $E_k$ coincides with the set
of $\mu$-measurable functions on $X$ which admit versions that are polynomials
of order $k$ in the usual algebraic sense (an algebraic polynomial of order $k$
is a function whose restriction to every straight line is a polynomial
of order~$k$).

The elements of $\mathcal{X}_2$ admit the following relatively simple
representation: for every $f\in \mathcal{X}_2$ there are numbers $c_n$ and an orthonormal sequence
$\{\xi_n\}\subset X_\mu^*$ such that $\sum_{n=1}^\infty c_n^2<\infty$ and
\begin{equation}\label{repr}
f=\sum_{n=1}^\infty c_n(\xi_n^2-1),
\end{equation}
where the series converges almost everywhere and in $L^2(\mu)$.
For any $k$, the elements of $\mathcal{X}_k$ can be represented by series in Hermite polynomials of order $k$
in the variables $\xi_n$. Also, in the case of the classical Wiener space, they can be written as
multiple Wiener integrals.
In that case, $X=C[0,1]$ or $X=L^2[0,1]$, $\mu$ is the Wiener measure,
its Cameron--Martin space $H$ is the space of all absolutely continuous functions $h$ on $[0,1]$
such that $h(0)=0$ and $h'\in L^2[0,1]$; $(u,v)_H=(u',v')_{L^2}$. Every element in $X_\mu^*$ can be written
as the Wiener stochastic integral
$$
\xi(x)=\int_0^1 u(t) dx(t),\quad u\in L^2[0,1].
$$
Letting
$$
h(t)=\int_0^t u(s)ds,
$$
we have $\xi=\widehat{h}$. Similarly, any element in $\mathcal{X}_2$ can be written as the double Wiener integral
$$
f(x)=\int\int q(t,s) dx(t) dx(s),
$$
where $q\in L^2([0,1]^2)$. However, the first integral $\xi(s,x)=\int q(t,s) dx(t)$
must be an adapted process (so that the second integral is already an It\^o integral
of an adapted process with respect to the Wiener process), i.e., for every $s$,
the random variable $\xi(s,x)$ must be
measurable with respect to the $\sigma$-field generated by the variables $x(t)$ with $t\le s$; for this reason
it is required that $q(t,s)=0$ whenever $t>s$.

Every function $f\in \mathcal{X}_2$ has the Malliavin gradient $D_Hf$ along $H$ that is a measurable
linear operator from $X$ to $H$; for $f$ of the form (\ref{repr}) we have
$$
D_Hf(x)=2\sum_{n=1}^\infty c_n\xi_n(x)e_n,
$$
where $\{e_n\}$ is an orthonormal sequence in $H$ such that $\xi_n=\widehat{e}_n$; without loss of generality
we may assume that $\{e_n\}$ is a basis in $H$.
Therefore, the second derivative $D_H^2f(x)$ is the symmetric Hilbert--Schmidt operator
with an eigenbasis $\{e_n\}$ and the corresponding eigenvalues~$\{2c_n\}$.
So the situation is similar to the case of $\mathbb{R}^d$, where the function
$Q(x)=\sum_{n=1}^d c_n (x_n^2-1)$ has the gradient $\nabla Q(x)=2\sum_{n=1}^d c_n x_ne_n$.
In the coordinate-free form $Q(x)=(Ax,x)-{\rm trace}\, A$, where $A$ is a symmetric operator,
$\nabla Q(x)=2Ax$.
The only difference is that the series of $c_n \xi_n^2$ does not have a separate meaning
unless the series of $|c_n|$ converges, so a typical element of $\mathcal{X}_2$ just formally
looks like ``a quadratic form minus a constant''.

It is clear that $D_Hf=2\widehat{D_H^2f}$. Conversely, for any symmetric Hilbert--Schmidt operator $A$
that has an eigenbasis $\{e_n\}$ and eigenvalues~$\{a_n\}$, there is $f\in \mathcal{X}_2$ of the
form (\ref{repr}) such that $A=D^2_Hf(x)$ and $\widehat{A}x=D_Hf(x)$.

It is readily seen that in the case of the classical Wiener space and an element $f\in \mathcal{X}_2$
represented by means of a double Wiener integral with a kernel~$q$ one has
$$
(D_Hf(x),h)_H=\int_0^1 \int_0^1 [q(t,s)+q(s,t)]h'(t) dt dx(s), \quad h\in H .
$$
For the second derivative we have
$$
(D_H^2f(x)h_1,h_2)_H=\int_0^1 \int_0^1 [q(t,s)+q(s,t)]h_1'(t)h_2'(s) dt ds,
\quad h_1, h_2\in H .
$$

Now we may ask whether two elements $f$ and $g$ in $\mathcal{X}_2$ are linearly dependent if
their gradients $D_Hf(x)$ and $D_Hg(x)$ are linearly dependent for almost all~$x$;
what about $k$ elements in $\mathcal{X}_2$? In terms of the second derivatives along $H$
our question is this: if two symmetric Hilbert--Schmidt operators $A$ and $B$ on $H$ are such
that $\widehat{A}x$ and $\widehat{B}x$ are linearly dependent for almost all~$x$,
is it true that the operators $A$ and $B$ are linearly dependent?
The same question can be asked about not necessarily symmetric operators on~$H$ and also
about more general measurable linear operators (as well as about more than two such objects).

For example, given two quadratic forms $(Ax,x)$ and $(Bx,x)$
on $\mathbb{R}^d$ with symmetric operators $A$ and $B$, the linear dependence of the forms
(or, what is the same, the linear dependence of the corresponding
elements $(Ax,x)-{\rm trace}\, A$ and $(Bx,x)-{\rm trace}\, B$ of $\mathcal{X}_2$)
is equivalent to the linear dependence of the operators, and both follow from
the condition that $Ax$ and $Bx$ are linearly dependent for $x$ in a positive measure set.
We are concerned with infinite-dimensional generalizations of this fact.

\begin{lemma}
Two functions $f,g\in \mathcal{X}_2$ are linearly dependent precisely when
the operators $D_H^2f$ and $D_H^2g$ on $H$ are linearly dependent.
\end{lemma}
\begin{proof}
If $\alpha  f(x)+\beta g(x)=0$ for some numbers $\alpha,\beta$, then obviously
$\alpha  D_H^2f+\beta D_H^2g=0$. Suppose the latter equality holds.
This means that the $H$-valued mapping $\alpha D_Hf+\beta D_Hg$ has zero derivative along~$H$.
It follows that $\alpha D_Hf(x)+\beta D_Hg(x)$ is a constant vector $h_0\in H$.
Therefore, $\alpha  f(x)+\beta g(x)=\widehat{h_0}(x)+c$, where $c$ is a constant.
Since $f$ and $g$ are orthogonal in $L^2(\mu)$ to all elements in $E_1$, we conclude
that $c=0$ and $h_0=0$.
\end{proof}

We recall the following zero-one law: for every $\mu$-measurable linear
subspace $L\subset X$ one has either $\mu(L)=0$ or $\mu(L)=1$.
There is also a similar zero-one law for measurable polynomials~$\psi$:
the set $\{x\colon\, \psi(x)=0\}$ has measure either $0$ or $1$,
see \cite[Theorem 3.2.10 and Proposition 5.10.10]{B1}.

\begin{theorem}\label{t1}
Let $A_1,\ldots,A_k$ be linearly independent Hilbert--Schmidt operators
on $H$. Then either the vectors $\widehat{A_1}x,\ldots,\widehat{A_k}x$
are linearly independent a.e. or there is a finite-dimensional bounded operator $D$
of rank at most $k-1$ that is a nontrivial linear combination of $A_1,\ldots,A_k$.
\end{theorem}
\begin{proof}
Let $k=2$.
We may assume that $\widehat{A_1}x\not=0$ a.e., since otherwise
$\widehat{A_1}x=0$ a.e. by the zero-one law, hence $A_1=0$.
Suppose that $\widehat{A_1}x,\widehat{A_2}x$ are linearly dependent
on a positive measure set.
By the zero-one law for polynomials they are linearly dependent a.e.,
because the set $Z$ of points $x$ at which
$\widehat{A_1}x$ and $\widehat{A_2}x$ are linearly dependent is characterized
by the equality
$(\widehat{A_1}x, \widehat{A_2}x)_H^2
-(\widehat{A_1}x, \widehat{A_1}x)_H(\widehat{A_2}x, \widehat{A_2}x)_H=0$.
Hence there is a function $k$ on $X$ such that
$$
\widehat{A_2}x=k(x)\widehat{A_1}x \ \hbox{a.e.}
$$
Then $k(x)=(\widehat{A_1}x, \widehat{A_2}x)_H/|\widehat{A_1}x|_H^2$.
The functions
$$
P(x)=(\widehat{A_1}x, \widehat{A_2}x)_H,
\quad
Q(x)=|\widehat{A_1}x|_H^2
$$
are obviously differentiable along $H$. Note that the derivatives of $\widehat{A}_1$
and $\widehat{A}_1$ along $H$ are just the initial operators $A_1$ and $A_2$, respectively.
Differentiating the equality
$$
Q(x)\widehat{A_2}x=P(x)\widehat{A_1}x,
$$
we get  almost everywhere
$$
Q(x)A_2+ D_HQ(x)\otimes \widehat{A_2}x=P(x)A_1+ D_H P(x)\otimes \widehat{A_1}x,
$$
where for any two vectors $u,v\in H$ the operator $u\otimes v$ is
defined by $u\otimes v(h)=(u,h)_H v$.
It should be noted that differentiating is possible almost everywhere, since
for almost every $x\in Z$ the set $Z$ contains all straight lines $x+\mathbb{R}h$
for every vector $h\in H$ that is a linear combination of the elements of
a fixed orthonormal basis $\{e_i\}$ in $H$ with rational coefficients.

Let $x$ be any point where
the previous equality holds and $Q(x)\not=0$.
Then
$$
A_2 =k(x)A_1+ Q(x)^{-1}[D_H P(x)-k(x)D_HQ(x)] \otimes \widehat{A_1}x.
$$
Setting $c:=k(x)$ and
$$
D:=Q(x)^{-1}[D_H P(x)-k(x)D_HQ(x)] \otimes \widehat{A_1}x,
$$
we arrive at the identity $A_2=cA_1+D$, where $D$ has rank at most~$1$.

Suppose now that our assertion is true for some $k>2$ and consider linearly
independent operators $A_1,\ldots,A_k,A_{k+1}$. We may assume that
$\widehat{A_1}x,\ldots,\widehat{A_k}x$ are linearly independent a.e.
If $\widehat{A_1}x,\ldots,\widehat{A_{k+1}}x$ are linearly dependent on a positive
measure set, then they are linearly dependent a.e., hence there are functions
$c_1,\ldots,c_k$ on $X$ such that a.e.
$$
\widehat{A_{k+1}}x=c_1(x)\widehat{A_{1}}x+\cdots+c_k(x)\widehat{A_{k}}x.
$$
It is readily verified again that the functions $c_i$ (which can be found explicitly)
are differentiable along $H$ a.e.
Differentiating we arrive at the equality a.e.
$$
A_{k+1}=c_1(x)A_1+\cdots+c_k(x)A_k+D(x),
$$
where $D(x)$ is a sum of $k$ one-dimensional operators, hence has rank at most $k$.
It remains to take $x$ as a common point of differentiability of $c_1,\ldots,c_k$.
\end{proof}

\begin{remark}
{\rm
It is known (see, e.g., \cite{Y}) that for any measurable linear
operator $T$
from a separable Fr\'echet space $X$ with a Gaussian measure $\mu$
(actually, not necessarily Gaussian)
to a separable Banach space $Y$
there is a separable reflexive Banach space $E$ compactly embedded into $X$
and having full measure such that $T$ coincides almost everywhere with a bounded
linear operator from $E$ to $Y$. Using  this result, one can reduce the previous theorem
(in the case  of Fr\'echet spaces)
to the case of bounded operators.
}\end{remark}

\begin{corollary}
Let $k=2$ and let $A_1$ and $A_2$ be symmetric. If $\widehat{A_1}x$ and
$\widehat{A_2}x$ are linearly dependent for vectors $x$ in a set of positive measure,
then $A_1$ and $A_2$ are linearly dependent.
\end{corollary}
\begin{proof}
Let us return to the proof of the theorem, where
$\widehat{A_2}x=k(x)\widehat{A_1}x$ a.e. with some measurable function $k$.
We also have $A_2=cA_1+D$, where $D$ has rank at most $1$. Suppose that
$D\not=0$. Clearly, $D$ is also symmetric, hence
$\widehat{D}=\lambda\widehat{h_1}h_1$ for some nonzero vector $h_1\in H$ and
a nonzero number $\lambda$.
Since  $\widehat{D}x=(k(x)-c)\widehat{A_1}x$ a.e.
and $\widehat{D}x\not=0$ a.e. (otherwise $D=0$), we
see $\widehat{A_1}$ takes its values in the
one-dimensional range of $D$ a.e. This means that both
$A_1$ and $A_2$ are one-dimensional and by their symmetry are
of the form $h\mapsto c_i(h,h_1)_H h_1$ with some constants $c_1$ and $c_2$,
whence it is obvious that they are linearly dependent.
\end{proof}

\begin{corollary}
Suppose that functions $f_1,\ldots,f_k$, where $k>2$, belong to $\mathcal{X}_2$ and are linearly
independent. Then either they have a joint density of distribution or some nontrivial
linear combination $c_1f_1+\cdots+c_nf_n$ is a degenerate element of $\mathcal{X}_2$ of rank~$k-1$,
i.e., a second order polynomial in $k-1$ elements of the space $\mathcal{X}_1$
of measurable linear functionals.
\end{corollary}

\begin{corollary}
Let $\mu$ be a nondegenerate centered Gaussian measure on a separable Fr\'echet space $X$,
let $H$ be its Cameron--Martin space~$H$, and let $A_1,\ldots,A_k$ be linearly independent
 continuous linear operators on $X$ with values in a Banach space~$E$.
Then either the vectors $A_1x,\ldots,A_kx$ are linearly independent a.e. or some
nontrivial linear combination $c_1A_1+\cdots+c_kA_k$ has rank at most $k-1$.
\end{corollary}
\begin{proof}
It suffices to consider the case of infinite-dimensional $H$.
We may assume that $E$ is separable. Embedding $E$ into $l^2$ by means of an injective
continuous linear operator we can pass to the case
$E=H$. Finally, taking an injective Hilbert--Schmidt operator $S$ and dealing with
the operators $SA_1,\ldots,SA_k$ we arrive at the situation in the theorem
(in fact, there is no need to take $S$, since the restrictions of our operators to $H$
will be automatically Hilbert--Schmidt operators, see \cite[Proposition~3.7.10]{B1}).
\end{proof}

Let us show that the rank of a degenerate linear combination indicated in Theorem~\ref{t1}
cannot be made smaller in general.

\begin{proposition}
For every $k\in\mathbb{N}$ there exist  operators $A_1,\ldots,A_k\colon\, \mathbb{R}^k
\to \mathbb{R}^d$, where $d=k(k-1)/2$, such
that, for every vector $x\in\mathbb{R}^k$, the vectors $A_1 x,\ldots,A_kx$ are linearly dependent, but for
every nonzero vector $x=(x_1,\ldots,x_k)$ the operator $\sum_{i=1}^k x_iA_i$ has rank $k-1$.
In particular, a nontrivial linear combination of $A_1,\ldots,A_k$ cannot be of rank less than $k-1$.
\end{proposition}
\begin{proof}
Let $A_1,\ldots,A_k$ be certain operators from $\mathbb{R}^k$
to $\mathbb{R}^d$. Each operator $A_i$ is represented by a matrix $(a_i^{l,j})_{l\le d, j\le k}$ with
$k$ columns $A_i^1,\ldots,A_i^k$, where $A_i^j=(a_i^{1,j},\ldots,a_i^{d,j})$.
We first choose matrices $A_i$ such that $A_i^i=0$ and
$$
\sum_{i=1}^k x_iA_i x=0 \quad \forall\, x\in\mathbb{R}^k.
$$
The latter is equivalent to the following system
of vector equations:
$$
A_i^j=-A_j^i, \ i\not=j.
$$
 Thus, we obtain  $d=k(k-1)/2$ vector equations for $k(k-1)$ nonzero columns of the matrices $A_i$.
Therefore, once we define $d$ columns (one column for every equation), the remaining
columns will be uniquely determined from the equations.

Let us consider two cases: $k$ is odd or $k$ is even.
Let $k$ be odd, $m=(k-1)/2$. In every matrix $A_i$ we define $m$ columns
(the remaining columns  will be determined from the equations) in the following way:
$A_i^j=e_{(i-1)m+j-i}$, $j=i+1,\ldots, i+m\mod k$, where $e_1,\ldots,e_d$ is
the standard basis of $\mathbb{R}^d$ and $j\mod k$ means the integer number $r\le k$ with
$j=pk+r$. It is easily seen that for every $s\in\{1,\ldots,d\}$
 there exists a unique pair of columns $A_i^j,A_j^i$
  with $a_i^{sj}\not=0$; for different $s$ such pairs  are different.
 Now let us fix a nonzero vector $x\in \mathbb{R}^k$.
 Let $y\in {\rm Ker}\, \sum_{i=1}^k x_iA_i$.
  We obtain that $x_iy_j-x_jy_i=0$ for all pairs $i,j$, which
  yields the linear dependence of the vectors $x$ and $y$.
  Therefore, $\dim\Bigl( {\rm Ker} \sum_{i=1}^k x_iA_i\Bigr)=1$, which means
that the rank of   $\sum_{i=1}^k x_iA_i$ is $k-1$.

For example, let $k=3$. Then $d=k(k-1)/2=3$, $m=(k-1)/2=1$.
We have $k(k-1)/2=3$ equations $A_i^j=-A_j^i$, $i\neq j$.
First we define one column in every matrix, which along with the
equality $A_i^i=0$ yields the following representation:
$$\begin{pmatrix}
0&1&A_{1}^{13}\\
0&0&A_{1}^{23}\\
0&0&A_{1}^{33}
\end{pmatrix},
\begin{pmatrix}
A_2^{11}&0&0\\
A_2^{21}&0&1\\
A_2^{31}&0&0
\end{pmatrix},
\begin{pmatrix}
0&A_3^{12}&0\\
0&A_3^{22}&0\\
1&A_3^{32}&0
\end{pmatrix}.
$$
From the equations we find that
$$
A_1^3=-A_3^1=(0,0,-1),
\
A_2^1=-A_1^2=(-1,0,0),
\
A_3^2=-A_2^3=(0,-1,0),
$$
so that our matrices are
$$\begin{pmatrix}
0&1&0\\
0&0&0\\
0&0&-1
\end{pmatrix},
\begin{pmatrix}
-1&0&0\\
0&0&1\\
0&0&0
\end{pmatrix},
\begin{pmatrix}
0&0&0\\
0&-1&0\\
1&0&0
\end{pmatrix}.
$$
Suppose now we have a nontrivial linear combination
$$
A=x\begin{pmatrix}
0&1&0\\
0&0&0\\
0&0&-1
\end{pmatrix}+
y\begin{pmatrix}
-1&0&0\\
0&0&1\\
0&0&0
\end{pmatrix}+
z\begin{pmatrix}
0&0&0\\
0&-1&0\\
1&0&0
\end{pmatrix}
$$
and a vector $y=(a,b,c)\in {\rm Ker}\, A$. Then
$$
A\begin{pmatrix}
a\\
b\\
c
\end{pmatrix}=
\begin{pmatrix}
-ay+bx\\
yc-zb\\
-xc+za
\end{pmatrix}=
\begin{pmatrix}
0\\
0\\
0
\end{pmatrix}.
$$
Hence the vectors $(a,b,c)$ and $(x,y,z)$ are linearly dependent, so
 $\dim {\rm Ker}\, A=1$ and the rank of $A$ is~$2$.

In case of even $k$ everything is similar.
  Let $k=2m$. In every matrix $A_1,\ldots,A_{k-1}$ we define $m$ columns in the
  following way. First, dealing with the columns $A_i^1,\ldots,A_i^{k-1}$
and ignoring the last $k-1$ lines we
  define the matrix elements in the same way as for
  the odd number $k'=k-1$. Next, in the $i$th matrix
  (excepting the last one) in the $k$th column we put $1$
  at the $(d-(k-1)+i)$th position and put $0$ at the other positions.
\end{proof}

\begin{example}
{\rm
In case $k=3$, we have a simple example of three linearly independent
elements  $xy$, $x^2-1$, $y^2-1$ of $\mathcal{X}_2$
on the plane with the standard Gaussian measure such that
their gradients $(y,x)$, $(2x,0)$ and $(0,2y)$ are linearly dependent at every point $(x,y)$
in the plane. Thus, their second derivatives are symmetric linearly independent operators whose
values on every fixed vector are linearly dependent.
}\end{example}

\section{Convergence in law in the second Wiener chaos}

Let us first recall some important known results.

\begin{theorem}\label{NP-double}
Let $g_1,g_2,\ldots$ be a sequence of independent $N(0,1)$
random variables. Let $F_n$ be a sequence of the form
$F_n=\sum_{k=1}^\infty \alpha_{k,n}(g_k^2-1)$
with $\sum_{k=1}^\infty \alpha_{k,n}^2=\frac12$, i.e., ${\rm E}[F_n^2]=1$.
Assume also that $\sup_{k\ge  1}|\alpha_{k,n}|\to 0$
as $n\to\infty$. Then $F_n\overset{\rm law}{\to} N(0,1)$
as $n\to\infty$.
\end{theorem}

This result is a very special case of a classical extension
of the Lindeberg theorem (see \cite[Section~21]{GK} or
\cite[Section~21.2]{Loeve})
on convergence
of a sequence of series $F_n=\sum_{k=1}^\infty \xi_{k,n}$ with independent
in every series centered random variables such that $\sum_{k=1}^\infty {\rm E}\xi_{k,n}^2=1$.
The condition of convergence in law to the standard normal distribution is that
$$
\sum_{k=1}^\infty {\rm E}(\xi_{k,n}^2I_{|\xi_{k,n}|\ge \varepsilon})\to 0\quad \hbox{for every $\varepsilon>0$.}
$$
In our case this condition is obviously satisfied, since
\begin{eqnarray*}
{\rm E}[\alpha_{k,n}^2(g^2-1)I_{|g^2-1|\ge \frac{\varepsilon}{|\alpha_{k,n}|}}]&\le& \alpha_{k,n}^2\sqrt{{\rm E}[(g^2-1)^4]}
\sqrt{{\rm P}(|g^2-1|\geq \varepsilon/|\alpha_{k,n}|)}\\
&\le& C \frac{|\alpha_{k,n}|^3}{\varepsilon}
\le \frac{C}{\varepsilon} \sup_{k\geq 1}|\alpha_{k,n}|\times \alpha_{k,n}^2,
\end{eqnarray*}
where $g$ has the standard normal distribution.
It is worth noting that we could also use a recent result of
Nualart and Peccati \cite{NP} on convergence of distributions of multiple Wiener integrals.

We shall also need the following result of Peccati and Tudor \cite{PT}
(see also \cite[Chapter 6]{np-book}).

\begin{theorem}\label{PT-thm}
Let $k_1,k_2\in \mathbb{N}$ be fixed and let $k=k_1+k_2$.
Suppose we are given standard normal variables
$g_{i,j,n}$, $1\le i\le k$, $j,n=1,2,\ldots$, such that, for each fixed $n$,
the variables $g_{i,j,n}$ are jointly Gaussian
and ${\rm E}[g_{i,j,n}g_{i,j',n}]=0$ for any $i$ whenever $j\neq j'$.
Let us consider random vectors $F_n=(F_{1,n},\ldots,F_{k,n})\in\mathbb{R}^k$, where
$F_{i,n}=g_{i,1,n}$ if $i=1,\ldots,k_1$ and
$$
F_{i,n} = \sum_{j=1}^\infty \lambda_{i,j,n}(g_{i,j,n}^2-1) \quad\hbox{if $i=k_1+1,\ldots,k_1+k_2$,
where $\lambda_{i,j,n}\in\mathbb{R}, \ \sum_{j=1}^\infty \lambda_{i,j,n}^2<\infty$.}
$$
Assume that a finite limit $C_{i,j}:=\lim\limits_{n\to\infty}{\rm E}[F_{i,n}F_{j,n}]$
exists for all $i,j$. Then the following two assertions are equivalent{\rm:}

{\rm(a)} $(F_{1,n},\ldots,F_{k,n})\overset{\rm law}{\to} N_k(0,C)$  as $n\to\infty$,
where $C=(C_{i,j})_{1\le i,j\le k}${\rm;}

{\rm(b)} $F_{i,n}\overset{\rm law}{\to} N_1(0,C_{i,i})$  as $n\to\infty$
for each $i=k_1+1,\ldots,k_1+k_2$.
\end{theorem}

We now use Corollary \ref{NP-double} and Theorem \ref{PT-thm}
to prove Theorem~\ref{main2}.

\begin{proof}
We can assume that the limit of the distributions of $F_n$ is not
the Dirac mass at zero (otherwise the desired conclusion is trivial).
It is known that convergence in law for a sequence of
measurable polynomials of a fixed degree yields boundedness in all $L^p$
(see, e.g., \cite[Lemma~1]{AB}, \cite[Exercise~9.8.19]{B2} or \cite[Lemma 2.4]{NP}).
Therefore, without loss of generality we may assume that, for all $i$ and $n$,
$$
\sum_{k=1}\lambda_{i,k,n}^2=\frac12,
$$
that is, ${\rm E}[F_{i,n}^2]=1$ for all $i$ and $n$.
We can also assume that
$$
|\lambda_{i,1,n}|\ge  |\lambda_{i,2,n}|\ge  \cdots.
$$
Finally, using a suitable diagonalization {\it \`a la} Cantor,
we can assume that
$$
\lambda_{i,j,n}\to\mu_{i,j} \quad\hbox{as $n\to\infty$.}
$$
By Fatou's lemma we have $C_i:=\sum_{k=1}^\infty \mu_{i,k}^2\le \frac12$.
On the other hand, we claim that there exist natural numbers $D_{i,n}\to\infty$ such that
\begin{equation}\label{claim1}
C_{i,n}:=\sum_{k=1}^{D_{i,n}} (\lambda_{i,k,n}-\mu_{i,k})^2\to 0
\quad \hbox{as $n\to\infty$.}
\end{equation}
Indeed, let $a_{i,k,n}=(\lambda_{i,k,n}-\mu_{i,k})^2$. For every  $k\ge  1$,
let $B_{i,k}\ge  1$ be the smallest integer
such that if $n\ge  B_{i,k}$, then $a_{i,1,n}+\ldots+a_{i,k,n}\le \frac1k$.
It is clear that $(B_{i,k})_{k\ge  1}$ is an
increasing sequence. Without loss of generality, one can
assume that $B_{i,k}\to\infty$ as $k\to\infty$ (if $B_{i,k}\not\to\infty$, then
 $B_{i,k}=B_{i,\infty}$ for all $k$ large enough, which means
that $a_{i,1,n}+\ldots+a_{i,N,n}\le \frac1k$ for all $n\ge  B_{i,\infty}$
and all $N$; then $a_{i,j,n}=0$ for all $j$ and all $n\ge  B_{i,\infty}$, so that the existence of $D_{i,n}$
becomes obvious).
Set $D_{i,n}=\sup\{k\colon\, B_{i,k}\le n\}$. In particular, one has $B_{D_{i,n}}\le n$, implying in turn that
$a_{i,1,n}+\ldots+a_{i,D_{i,n},n}\le \frac1{D_{i,n}}$. Moreover, $D_{i,n}\uparrow\infty$ (since $n<B_{i,D_{i,n}}+1$;
if $D_{i,n}\not\to\infty$, then $D_{i,n}=D_{i,\infty}$ for $n$
large enough, which is absurd when $n\to\infty$).

It is clear from (\ref{claim1})  that
$$
\sum_{j=1}^{D_{i,n}}\lambda_{i,j,n}(g_{i,j,n}^2-1)
\overset{\rm (law)}{=}\sum_{j=1}^{D_{i,n}}\lambda_{i,j,n}(g_{i,j,1}^2-1)\overset{L^2}{\to} \sum_{j=1}^\infty \mu_{i,j}(g_{i,j,1}^2-1)\quad \mbox{as $n\to\infty$.}
$$
On the other hand, we claim that
\begin{equation}\label{nunupec}
\sum_{j=D_{i,n}+1}^\infty \lambda_{i,j,n}(g_{i,j,n}^2-1)\overset{\rm law}{\to} N(0,1-2C_i).
\end{equation}
Indeed, if $C_i=\frac12$, then
\begin{eqnarray*}
\sum_{j=1}^\infty (\lambda_{i,j,n}-\mu_{i,j})^2 &= &2\sum_{j=1}^\infty (\mu_{i,j}-\lambda_{i,j,n})\mu_{i,j}\\
&\le& 2\sum_{j=1}^N |\mu_{i,j}||\mu_{i,j}-\lambda_{i,j,n}|
+2\Bigl(\sum_{j=1}^\infty |\mu_{i,j}-\lambda_{i,j,n}|^2\Bigr)^{1/2}
\Bigl(\sum_{j=N+1}^\infty \mu_{i,j}^2\Bigr)^{1/2}\\
&\le& 2\sum_{j=1}^N |\mu_{i,j}||\mu_{i,j}-\lambda_{i,j,n}|+2\sqrt{2}
\Bigl(\sum_{j=N+1}^\infty \mu_{i,j}^2\Bigr)^{1/2},
\end{eqnarray*}
so that
$$
\limsup_{n\to\infty}
\sum_{j=1}^\infty (\lambda_{i,j,n}-\mu_{i,j})^2\le 2\sqrt{2}\Bigl(\sum_{j=N+1}^\infty \mu_{i,j}^2\Bigr)^{1/2},
$$
whence it follows that $\sum_{k=1}^\infty (\lambda_{i,j,n}-\mu_{i,j})^2\to 0$ as $n\to\infty$. As a result,
one has
$$
\sum_{j=D_{i,n}+1}^\infty \lambda_{i,j,n}^2\le 2
\sum_{j=1}^\infty (\lambda_{i,j,n}-\mu_{i,j})^2
+2\sum_{j=D_{i,n}+1}^\infty \mu_{i,j}^2\to 0
$$
and (\ref{nunupec}) is shown whenever $C_i=\frac12$.
Assume now that $C_i<\frac12$.
Since $D_{i,n}\to\infty$ and
$$
D_{i,n}\lambda^2_{i,D_{i,n}+1,n}\le \lambda_{i,1,n}^2 +\ldots + \lambda_{i,D_{i,n},n}^2\le \frac12,
$$
we obtain that $\lambda_{i,D_{i,n}+1,n}\to 0$ as $n\to\infty$.
Let us consider the variables
$$
G_{i,n} = \frac{1}{\sqrt{1-2C_{i,n}}}\sum_{j=D_{i,n}+1}^\infty \lambda_{i,j,n}(g_{i,j,n}^2-1).
$$
It is readily verified that ${\rm E}[G_{i,n}^2]=1$ for all $n$ and that
$$
\sup_{j\ge  D_{i,n}+1}\frac{|\lambda_{i,j,n}|}{\sqrt{1-2C_{i,n}}}\le \frac{|\lambda_{i,D_{i,n}+1,n}|}{\sqrt{1-2C_{i,n}}}\to 0\quad\mbox{as $n\to\infty$}.
$$
By Theorem \ref{NP-double} we obtain that
$G_{i,n}\overset{\rm law}{\to} N(0,1)$, which yields in turn that
(\ref{nunupec}) holds true whenever $C_i<\frac12$.

Extracting a subsequence, we may assume that a finite limit
\begin{equation}\label{existence}
\alpha_{i,i',j,j'}:=\lim_{n\to\infty} {\rm E}[g_{i,j,n}g_{i',j',n}]
\end{equation}
exists for all $i,i',j,j'$.
Note that $\alpha_{i,i,j,j'}=\delta_{jj'}$, where $\delta_{jj'}$ is the Kronecker symbol.
Let
$$
{\bf g}_\infty=(g_{i,j,\infty})_{i=1,\ldots,d,\,j= 1,2,\ldots}
$$
be a centered Gaussian family satisfying
${\rm E}[g_{i,j,\infty}g_{i',j',\infty}]= \alpha_{i,i',j,j'}$.
By (\ref{nunupec}), (\ref{existence}) and Theorem \ref{PT-thm},
we obtain (possibly, passing to a subsequence) that,  for any fixed $m\ge  1$,
as $n\to\infty$
\begin{equation}\label{PT}
(g_{1,1,n},\ldots,g_{1,m,n},\ldots,
g_{d,1,n},\ldots,g_{d,m,n},{\bf V}_n)
\overset{\rm law}{\to}
(g_{1,1,\infty},\ldots,g_{1,m,\infty},\ldots,
g_{d,1,\infty},\ldots,g_{d,m,\infty},\widehat{\bf g}),
\end{equation}
where $\widehat{\bf g}:=(\widehat{g}_{1},\ldots,\widehat{g}_{d})$ is
centered Gaussian and independent of ${\bf g}_\infty$. Let us now introduce some additional notation.
For any $n,m\ge  1$, set
\begin{eqnarray*}
{\bf U}_n&=&\biggl(\sum_{j=1}^{D_{1,n}} \lambda_{1,j,n}(g_{1,j,n}^2-1),\ldots,
\sum_{j=1}^{D_{d,n}} \lambda_{d,j,n}(g_{d,j,n}^2-1)\biggr),
\\
{\bf V}_n&=&\biggl(\sum_{j=D_{1,n}+1}^\infty \lambda_{1,j,n}(g_{1,j,n}^2-1),\ldots,
\sum_{j=D_{d,n}+1}^\infty \lambda_{d,j,n}(g_{d,j,n}^2-1)\biggr),
\\
{\bf W}_{m,n}&=&\biggl(\sum_{j=1}^{m} \mu_{1,j}(g_{1,j,n}^2-1),
\ldots,\sum_{j=1}^{m} \mu_{d,j}(g_{d,j,n}^2-1)\biggr)\quad\mbox{(here $n=\infty$ is possible)},
\\
{\bf Z}_{\infty}&=&\biggl(\sum_{j=1}^{\infty} \mu_{1,j}(g_{1,j,\infty}^2-1),
\ldots,\sum_{j=1}^{\infty} \mu_{d,j}(g_{d,j,\infty}^2-1)\biggr).
\end{eqnarray*}
As an immediate consequence of (\ref{PT}), we have
\begin{equation}
\lim_{n\to\infty} dist(({\bf W}_{m,n},{\bf V}_n);
({\bf W}_{m,\infty},\widehat{\bf g}))=0
\quad\mbox{for any fixed $m$},
\label{NR}
\end{equation}
where $dist$ stands for any distance that metrizes
convergence in probability (such as the Fortet--Mourier distance for instance).
On the other hand, since
$$
\lim_{n,m\to\infty} {\rm E}[\|{\bf U}_n-{\bf W}_{m,n}\|^2]=0, \quad
\lim_{m\to\infty} {\rm E}[\|{\bf W}_{m,\infty}-{\bf Z}_\infty\|^2]=0,
$$
where the usual norm in $\mathbb{R}^d$ is used,
we have
\begin{equation}\label{lim2}
\lim_{n,m\to\infty}dist(({\bf U}_n,{\bf V}_n);({\bf W}_{m,n},{\bf V}_n))=0, \quad
\lim_{m\to\infty}dist(({\bf W}_{\infty,m},\widehat{\bf g});({\bf Z}_\infty,\widehat{\bf g}))=0.
\end{equation}
By combining (\ref{NR}) and (\ref{lim2}) with
\begin{eqnarray*}
dist(({\bf U}_n,{\bf V}_n);({\bf Z}_\infty,\widehat{\bf g}))&\le&
dist(({\bf U}_n,{\bf V}_n);({\bf W}_{m,n},{\bf V}_n))+dist(({\bf W}_{m,n},{\bf V}_n);({\bf W}_{m,\infty},\widehat{\bf g}))\\
&&+dist(({\bf W}_{m,\infty},\widehat{\bf g});({\bf Z}_\infty,\widehat{\bf g})),
\end{eqnarray*}
we get that
$
\lim\limits_{n\to\infty}dist(({\bf U}_n,{\bf V}_n);({\bf Z}_\infty,\widehat{\bf g}))=0,
$
which yields the desired conclusion.
\end{proof}

It should be noted that Theorem~\ref{main2} generalizes the Arcones theorem to the
multidimensional case, but it is not known whether in the one-dimensional case the set of
distributions of measurable polynomials of a fixed degree $k>2$ is closed.

\bigskip

{\bf Acknowledgement}.
This work was supported by the RFBR projects
12-01-33009  and 11-01-12104-ofi-m, as well as 
the (french) ANR grant `Malliavin, Stein and Stochastic Equations with Irregular
Coefficients' [ANR-10-BLAN-0121].

\end{document}